\documentclass[11pt,psamsfonts]{amsart}
\usepackage{amsmath}
\usepackage{amsthm}
\usepackage{amssymb}
\usepackage{amscd}
\usepackage{amsfonts}
\usepackage{amsbsy}
\usepackage{graphicx}
\usepackage[dvips]{psfrag}
\usepackage{array}
\usepackage{color}
\usepackage{epsfig}
\usepackage{url}

\newcommand{\R}{\ensuremath{\mathbb{R}}}

\newtheorem {theorem} {Theorem}

\newtheorem {lemma}  [theorem]{Lemma}

\numberwithin{equation}{section}

\title[Dynamics and unfolding of piecewise  quasi--homogeneous  system]
{Global dynamics and unfolding of planar piecewise smooth quadratic quasi--homogeneous differential systems}

\author[ Y. TANG]{ Yilei  Tang$^{1, 2}$}

\address{$^1$ Center for Applied Mathematics and Theoretical Physics, University of Maribor, Krekova 2, Maribor,  SI-2000 Maribor, Slovenia}

\address{$^2$ School of Mathematical Science, Shanghai Jiao Tong University, Dongchuan Road 800, Shanghai, 200240, P.R. China}
\email{mathtyl@sjtu.edu.cn} 



\begin{document}

\begin{abstract}
 In this paper we research  global dynamics and bifurcations of  planar piecewise smooth quadratic quasi--homogeneous
 but non--homogeneous polynomial differential systems.
We  present sufficient and necessary conditions
 for  the existence of a center in  piecewise smooth quadratic  quasi--homogeneous  systems. Moreover, the center is global and non--isochronous if it exists,
which cannot appear in smooth quadratic  quasi--homogeneous  systems.
Then the global structures  of piecewise smooth quadratic  quasi--homogeneous but non--homogeneous  systems are studied.
Finally we investigate limit cycle bifurcations of the piecewise smooth quadratic quasi--homogeneous center and
give  the  maximal number of limit cycles bifurcating from the periodic orbits of the center
by applying the  Melnikov method for piecewise smooth near-Hamiltonian systems.

\end{abstract}

\maketitle

\section{ Introduction }\label{s1}

Since Andronov {\it et al} \cite{AVK} researched the properties of solutions of piecewise linear differential
systems,  there are lots of works in mechanics, electrical engineering and the theory of automatic control which are described
by non-smooth systems; see for  the works of Filippov \cite{Filippov},  di Bernardo {\it et al} \cite{BBCK},
Makarenkov and Lamb \cite{ML} and the references therein.

For the planar piecewise smooth linear differential systems  separated by a straight line,
\cite{Buz, HZ, LP} studied the  systems having two or three limit
cycles respectively.
More investigations of limit cycle bifurcations from linear  piecewise
differential systems can be seen in \cite{Freire12, GP}. 
The discussion of limit cycle bifurcations  in nonlinear piecewise differential equations has also been researched in many works;
see for instance    \cite{BBCKNTP, ChenR2015, LiaHR2012, WZ2016, ZKB2006}.  

However, there are seldom works giving completely global dynamics of  piecewise smooth  nonlinear differential systems.
Even for smooth polynomial differential systems there are only few classes   whose global structures were completely characterized,
as shown in  \cite{DumoLA2006, Reyn2007}.

A real planar polynomial differential system
\begin{eqnarray}  \label{quasi5-general}
 \dot{x}¢« = P(x, y), \qquad\qquad  \dot{y}¢« = Q(x, y),
\end{eqnarray}
is called a {\it quasi--homogeneous
polynomial differential system} if there exist constants
$s_1,s_2,d\in \mathbb{Z}_+$ such that for an arbitrary
$\alpha\in
\R _+$ it holds that
\begin{equation*}  \label{alphaPQ}
P(\alpha^{s_1}x, \alpha^{s_2}y) =
\alpha^{s_1+d-1} P(x, y), \ \  \ \  Q(\alpha^{s_1}x, \alpha^{s_2}y)=
\alpha^{s_2+d-1}  Q(x, y),
\end{equation*}
where  $PQ \not\equiv 0$, $P(x, y), Q(x, y) \in \R [x, y]$, $\mathbb{Z}_+$ is the set of positive integers and $\mathbb R_+$
is the set of positive real numbers. We denominate
$w =(s_1, s_2, d)$ the {\it weight vector} of system (\ref{quasi5-general}) or of its associated vector field.
When $s_1=s_2=1$, system (\ref{quasi5-general}) is a homogeneous one of degree $d$.
Clearly, quasi--homogeneous  system (\ref{quasi5-general})  has a unique  {\it minimal weight vector} (MWF for short)
$ \widetilde{w}=(\widetilde{s}_1 , \widetilde{s}_2, \tilde{d})$  satisfying  $\tilde{s}_1\le s_1, \tilde{s}_2\le s_2$ and $\tilde{d}\le d$
for any other weight  vector $(s_1, s_2, d)$ of system (\ref{quasi5-general}).
We say that  system (\ref{quasi5-general}) has {\it degree $n$} if $n=\max\{\mbox{deg}\,P,\mbox{deg}\, Q\}$.
In what follows we assume without loss of generality
that $P$ and $Q$ in system \eqref{quasi5-general} have not a non--constant common factor.

%

Smooth Quasi--homogeneous polynomial differential systems have  been
intensively studied by a great deal of authors  from different views.
We refer readers  to see for example the  integrability
\cite{Alga2009,   GGL2013, Gori1996, Hu2007,  Llib2002},
  the centers and limit cycles  \cite{Alga2012, Gavr2009, GinGL2015, Li2009},
 the  algorithm to compute quasi--homogeneous   systems with a given degree \cite{Garc2013},
the characterization of centers or topological phase portraits for quasi--homogeneous equations of degrees $3$-$5$ respectively
\cite{Aziz2014, Liang2014, TWZ2015}  and the references therein.

A real planar piecewise smooth polynomial differential system
\begin{eqnarray}  \label{quasi-piece}
\begin{array}{ll}
 \dot{x}¢« = P^+(x, y), \qquad  \dot{y}¢« = Q^+(x, y), \qquad y\ge 0,
 \\
 \dot{x}¢« = P^-(x, y), \qquad \dot{y}¢« = Q^-(x, y), \qquad y<0
\end{array}
\end{eqnarray}
is called a {\it piecewise smooth quasi--homogeneous
polynomial differential system} with two zones separated by the $x$-axis if both $(P^+(x, y), Q^+(x, y))$
and $(P^-(x, y)$, $Q^-(x, y))$  are quasi--homogeneous
polynomial vector fields.

In this paper we research the global dynamics and bifurcations of all piecewise smooth quadratic quasi--homogeneous but non--homogeneous  differential systems.
 First the existence of a global and non-isochronous center at the origin
 of  piecewise smooth quadratic quasi--homogeneous but non--homogeneous  systems is proved.
Notice that the origin of smooth quadratic quasi--homogeneous   systems
cannot be a center. 
 Then  we  characterize the global phase portraits of  piecewise smooth quadratic  quasi--homogeneous but non--homogeneous polynomial
vector fields. At last we perturb the piecewise smooth quadratic quasi--homogeneous system at the  center
by  generic  piecewise   polynomials of degree $n$,
and determine the maximal number of limit cycles bifurcating from the periodic orbits of the center
by using the first order  Melnikov function.

This article is organized as follows. In section \ref{s2}
we  prove that only one class of  piecewise smooth quadratic quasi--homogeneous but non--homogeneous differential systems has a center
  at the origin, and it is global and non-isochronous if it exists. Section \ref{s3} will concentrate on global structures and phase portraits of
piecewise  smooth quadratic quasi--homogeneous but non--homogeneous differential systems. The unfoldings and bifurcations of these systems are investigated
 for some critical values of parameters.
The last section  is devoted to   the limit cycle  bifurcations from the periodic orbits
of the piecewise  smooth quadratic quasi--homogeneous center.



\bigskip


\section{Center  of piecewise smooth quadratic quasi--homogeneous  systems} \label{s2}



Due to Proposition 17  of Garc\'ia, Llibre and P\'erez del R\'io \cite{Garc2013}, a smooth quasi-homogeneous but non-homogeneous quadratic system  has one of the following three forms:
\begin{align*}
& (i): ~ \dot{x}=a_1 y^2,\quad \dot{y}=b_1 x  ~\mbox{ with MWV}~ (3,2,2) ~\mbox{and}~  a_1b_1 \ne 0,
\\
& (ii): ~  \dot{x}=a_2 x y,\quad \dot{y}=b_{21} x+ b_{22} y^2 ~\mbox{ with MWV} ~(2,1,2) ~\mbox{and}~ a_2b_{21}b_{22} \ne 0,
\\
& (iii): ~  \dot{x}=a_{31}x+a_{32} y^2,\quad \dot{y}=b_3 y ~\mbox{ with MWV} ~(2,1,1) ~\mbox{and}~ a_{31} a_{32}b_3 \ne 0.
\end{align*}

Thus after taking appropriate linear changes of variable  $x$ together with a time  scaling,   we have three totally reduced
piecewise smooth quasi-homogeneous but non-homogeneous quadratic systems.

\begin{lemma}
\label{lm-forms}
Every planar piecewise smooth  quasi-homogeneous but non-homogeneous quadratic system is one of the following three systems:
\begin{equation*}
 \begin{array}{llll}
& (I): ~   &\dot{x}=a_1 y^2,\quad \dot{y}=b_1 x  &~\mbox{ if }~ y\ge 0,
\\
& &   \dot{x}=\tilde{a}_1 y^2,\quad \dot{y}= x  &~\mbox{ if }~ y < 0;
\\
& (II): ~  &\dot{x}=a_2 x y,\quad \dot{y}=b_{21} x+ b_{22} y^2 &~\mbox{ if } ~  y\ge 0,
\\
& &\dot{x}=\tilde{a}_2 x y,\quad \dot{y}=  x+  y^2 &~\mbox{ if } ~  y< 0;
\\
& (III): ~ & \dot{x}=a_{31}x+a_{32} y^2,\quad \dot{y}=b_3 y &~\mbox{ if } ~ y\ge 0,
 \\
& & \dot{x}=\tilde{a}_{31}x+ y^2,\quad \dot{y}=y &~\mbox{ if } ~ y< 0,
\end{array}
\end{equation*}
where all parameters cannot be zero.
\end{lemma}

\begin{proof}
From the transformation  $(x, y, dt) \to (x,y, \tilde{b}_1 dt)$, $(x, y, dt) \to ( \tilde{b}_{21}x
/\tilde{b}_{22}$, $y, \tilde{b}_{22}dt)$
  and $(x, y, dt) \to (\tilde{b}_3 x/\tilde{a}_{32}, y, \tilde{b}_3 dt)$, planar piecewise smooth quadratic quasi-homogeneous but non-homogeneous systems
\begin{equation*}
 \begin{array}{llll}
& (a): ~   &\dot{x}=a_1 y^2,\quad \dot{y}=b_1 x  &~\mbox{ if }~ y\ge 0,
\\
& &   \dot{x}=\tilde{a}_1 y^2,\quad \dot{y}=\tilde{b}_1 x  &~\mbox{ if }~ y < 0;
\\
& (b): ~  &\dot{x}=a_2 x y,\quad \dot{y}=b_{21} x+ b_{22} y^2 &~\mbox{ if } ~  y\ge 0,
\\
& &\dot{x}=\tilde{a}_2 x y,\quad \dot{y}= \tilde{b}_{21}  x+ \tilde{b}_{22} y^2 &~\mbox{ if } ~  y< 0;
\\
& (c): ~ & \dot{x}=a_{31}x+a_{32} y^2,\quad \dot{y}=b_3 y &~\mbox{ if } ~ y\ge 0,
 \\
& & \dot{x}=\tilde{a}_{31}x+\tilde{a}_{32} y^2,\quad \dot{y}=\tilde{b}_3 y &~\mbox{ if } ~ y< 0,
\end{array}
\end{equation*}
are changed into systems $(I)$, $(II)$ and $(III)$ respectively, where we still write $a_1/\tilde{b}_1$, $b_1/\tilde{b}_1$,
$\tilde{a}_1/\tilde{b}_1$, $a_2/\tilde{b}_{22}$, $b_{21}/\tilde{b}_{21}$, $b_{22}/\tilde{b}_{22}$, $\tilde{a}_2/\tilde{b}_{22}$,
$a_{31}/\tilde{b}_3$,  $a_{32}/\tilde{a}_{32}$, $b_3/\tilde{b}_3$ and $\tilde{a}_{31}/\tilde{b}_3$
as $a_1$, $b_1$, $\tilde{a}_1$, $a_2$, $b_{21}$, $b_{22}$, $\tilde{a}_2$, $a_{31}$,  $a_{32}$, $b_3$ and $\tilde{a}_{31}$  for simpler notations.
\end{proof}


\medskip


In the following, we briefly  present the Filippov convex method \cite{BBCK,Filippov,Kunze,KRG}
to study the dynamics of generic piecewise smooth quasi-homogeneous  system
\eqref{quasi-piece} close to the discontinuous line. 
This discontinuous line
\begin{eqnarray*}
\mathcal{L}:=\{(x,y)\in\mathbb{R}^2|~ F(x,y):=y=0\}
\end{eqnarray*}
 separates the  plane into two open nonoverlapping regions
\begin{eqnarray*}
Y^+=\{(x,y)\in\mathbb{R}^2| ~0<y\} ~~\mbox{and}~~Y^-=\{(x,y)\in\mathbb{R}^2|~ 0>y \}.
\end{eqnarray*}

Suppose that
\begin{eqnarray*}
\sigma(x,y)=\left<(F_x,F_y), (P^+, Q^+) \right> ~ \left<(F_x,F_y), (P^-, Q^-) \right>,
\label{sigma}
\end{eqnarray*}
where $<\cdot,\cdot>$ denotes the standard scalar product.
The {\it crossing set} can be defined by 
\begin{eqnarray}
\mathcal{L}_c:=\{(x,y)\in \mathcal{L}|~ \sigma(x,y)>0\},
\label{Lc}
\end{eqnarray}
indicating that at each point  of $\mathcal{L}_c$ the orbit of system \eqref{quasi-piece} crosses $\mathcal{L}$,
i.e., the orbit reaching $(x,y)$ from $Y^+$ (or $Y^-$) concatenates with the orbit entering $Y^-$
(or $Y^+$) from $(x,y)$.
The {\it sliding set} $\mathcal{L}_s$ is the complement of $\mathcal{L}_c$ in $\mathcal{L}$, which is defined as
\begin{eqnarray}
\mathcal{L}_s :=\{(x,y)\in \mathcal{L} | ~\sigma(x,y)\leq0\}.
\label{Ls}
\end{eqnarray}
Moreover, in $\mathcal{L}_s$ solving the equation
\begin{eqnarray}
\label{slid-point}
\left<(F_x, F_y),(P^--P^+, Q^--Q^+)\right>=0
\end{eqnarray}
we can obtain the singular  sliding points from the set of solutions.

Regarding to the piecewise smooth  system $(I)$, we can analyze that the crossing set and the sliding set in $\mathcal{L}$ are
\begin{eqnarray}
\label{I-c}
&\mathcal{L}_c^I=\{(x,y)\in \mathcal{L} | ~b_1 x^2>0\}=
\Big\{
\begin{array}{ll}
\{(x,y)\in \mathcal{L} | ~ x\ne 0\}  &~\mbox{if}~b_1>0,
\\
 \emptyset  &~\mbox{if}~b_1<0
\end{array}
 \end{eqnarray}
and
\begin{eqnarray}
& \mathcal{L}_s^I =\{(x,y)\in \mathcal{L} | ~b_1 x^2 \leq0\}=
\Big\{
\begin{array}{ll}
\{(x,y)\in \mathcal{L} | ~ x=0\}  &~\mbox{if}~b_1>0,
\\
 \mathcal{L}  &~\mbox{if}~b_1<0
\end{array}
\label{I-s}
 \end{eqnarray}
respectively by definitions \eqref{Lc} and \eqref{Ls}. Then, we find the only solution of \eqref{slid-point} for system $(I)$ in $\mathcal{L}_s^I$
is the origin, which is a singular sliding point and at the same time a boundary equilibrium because of the vanish of vector fields at the origin.

By an analogous computation of system  $(I)$, we have the crossing sets and the sliding sets in $\mathcal{L}$ of the forms
\begin{eqnarray}
\label{II-c}
&\mathcal{L}_c^{II}=\{(x,y)\in \mathcal{L} | ~b_{21} x^2>0\}=
\Big\{
\begin{array}{ll}
\{(x,y)\in \mathcal{L} | ~ x\ne 0\}  &~\mbox{if}~b_{21} >0,
\\
 \emptyset  &~\mbox{if}~b_{21}<0,
\end{array}
\\
 & \mathcal{L}_s^{II} =\{(x,y)\in \mathcal{L} | ~b_{21} x^2 \leq0\}=
\Big\{
\begin{array}{ll}
\{(x,y)\in \mathcal{L} | ~ x=0\}  &~\mbox{if}~b_{21}>0 ,
\\
 \mathcal{L}  &~\mbox{if}~b_{21}<0
\end{array}
\label{II-s}
 \end{eqnarray}
 and
\begin{equation}
\label{III-cs}
 \mathcal{L}_c^{III}=  \emptyset, \qquad  \mathcal{L}_s^{III} =\mathcal{L}
\end{equation}
for  piecewise smooth  systems $(II)$ and $(III)$, respectively.
We find  the origin of system $(II)$ in $\mathcal{L}_s^{II}$
is a unique singular sliding point, which is a boundary equilibrium.
Moreover, the discontinuous line $\mathcal{L}$ is full of non-isolated singular sliding points for system $(III)$, since  equation
\eqref{slid-point} always holds  on the sliding set $ \mathcal{L}_s^{III}$.


\smallskip

Notice that all smooth quadratic quasi--homogeneous  systems $(i)-(iii)$  have no centers,  since there exists an invariant line or an invariant curve passing through
the origin of such  systems. However, for piecewise smooth  quadratic quasi--homogeneous  systems we will find
the existence of a center at the origin under some parameter conditions.

An equilibrium of the piecewise smooth   system \eqref{quasi-piece} is called a {\it center} if all solutions sufficiently closed to it are periodic.
If all  periodic solutions inside the period annulus of the center have the same period it is said that the center is {\it isochronous}.
A center is called a {\it global center } when the periodic orbits
surrounding the center fill the whole plain except the center
itself.

\begin{theorem}\label{th-center}
 Piecewise smooth  quadratic quasi--homogeneous  systems  $(II)$ and $(III)$ have no centers on the phase space.
 Piecewise smooth  quadratic quasi--homogeneous  system  $(I)$ has a center at the origin if and only if  $a_1<0, b_1>0$ and $\tilde{a}_1>0$,
 which is global but not isochronous.
\end{theorem}

\begin{proof}
Notice that no equilibria of piecewise smooth systems $(I)$-$(III)$ exist in the regions $Y^{\pm}$.
From above mentioned analysis of sliding sets and singular sliding points, on the discontinuous line $\mathcal{L}$ systems $(I)$ and $(II)$
 have a unique singular sliding point at the origin, and $\mathcal{L}$ is full of non-isolated singular sliding points for system $(III)$.
Thus,  system $(III)$ has no centers on the plane.
It is easy to see that  system   $(II)$  has an invariant line $x=0$ passing through its origin,
yielding that the origin cannot be a center.
We only need to check whether the origin of system $(I)$ can be a center.

The corresponding smooth quadratic quasi--homogeneous  system  $(i)$ of piecewise smooth system $(I)$  has a double vanished eigenvalue and
by \cite[Theorem 3.5]{DumoLA2006}
the equilibrium at the origin is a cusp. Moreover, when $a_1b_1>0$  the piecewise smooth system $(I)$
has an invariant curve $\frac{a_1}{3}y^3 -\frac{b_1}{2} x^2=0$ passing through the origin $O_1:(0,0)$ in the half plane $Y^+$, and
when $\tilde{a}_1<0$  system $(I)$
has an invariant curve $\frac{\tilde{a}_1}{3}y^3 -\frac{1}{2} x^2=0$ passing through the origin $O_1$ in the half plane $Y^-$.
Therefore, only when a crossing  deleted neighborhood $(-\delta_0, \delta_0)\setminus \{0\}$ exists for small $\delta_0>0$
under the condition $a_1b_1<0$, $\tilde{a}_1>0$,
the origin $O_1$ of system  $(I)$ is possibly a center. Thus we get the parameter condition $a_1<0, b_1>0$ and $\tilde{a}_1>0$
from the expression of the crossing set $\mathcal{L}_c^I$ of piecewise smooth  system $(I)$.

When $a_1<0, b_1>0$ and $\tilde{a}_1>0$, the crossing set $\mathcal{L}_c^I$ is the $x$-axis except the origin
and the orbits surrounding the origin are spirals rotating  anti-clockwise.
Let $p^+(r, \theta)$ (resp. $p^-(r, \theta)$) be the solution of piecewise smooth  system $(I)$ in polar coordinates
$(x, y) = (\rho\cos\theta,  \rho\sin\theta)$ for $0\le \theta<\pi$ (resp. $-\pi\le \theta<0$),   satisfying that the initial condition $p^+(r,0)=r$ (resp.
$p^-(r,-\pi)=r$) holds,
which is well defined in the region $\mathbb{R}^2\setminus \mathcal{L}_s^I$.
Then, we  define the positive Poincar\'e half-return map as
 $\mathcal{P}^+(r):=\lim_{ \theta\to \pi} p^+(r,\theta)$
 and the negative Poincar\'e half-return map as $\mathcal{P}^-(r):=\lim_{ \theta\to 0} p^-(r, \theta)$, 
 as shown in Figure \ref{Center1}.
 The  Poincar\'e return map associated to  piecewise smooth  system $(I)$ is given by the composition of
these two maps
\begin{equation}\label{poincare}
\mathcal{P}_I(r):=\mathcal{P}^-(\mathcal{P}^+(r)).
\end{equation}
In order to obtain the existence of a center and further a global center at the origin,  via the definition \eqref{poincare} we need to prove
$\mathcal{P}_I(r)-r\equiv 0$ for arbitrary $r>0$.

\begin{figure}[htb]
\centering
\includegraphics[scale=.60]{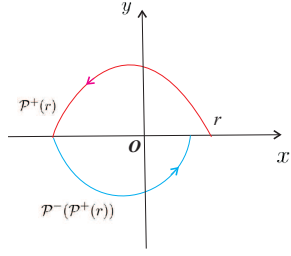}
\caption{  Existence of closed orbits for system $(I)$.  } 
\label{Center1}
\end{figure}

Piecewise smooth system $(I)$ has a polynomial first integral
$
H_1^+(x, y)=\frac{a_1}{3}y^3 -\frac{b_1}{2} x^2
$
if $y\ge 0$, and a first integral
$
H_1^-(x, y)=\frac{\tilde{a}_1}{3}y^3 -\frac{1}{2} x^2
$
if $y< 0$.
Then we have
$
H_1^+(r, 0)=H_1^+(\mathcal{P}^+(r), 0),
$
yielding that $\mathcal{P}^+(r)=r$. Furthermore, by
$$
H_1^-(\mathcal{P}^+(r), 0)=H_1^-(\mathcal{P}_I(r), 0)
$$
we get  $\mathcal{P}_I(r)=r$, implying that the solution curve of  piecewise smooth  system $(I)$  through $(r,0)$ is a closed orbit for arbitrary $r>0$.
Notice that the origin $O_1$ is the unique
singularity  of piecewise smooth  system $(I)$ when  $a_1<0, b_1>0$ and $\tilde{a}_1>0$. Therefore,
the origin $O_1$ of piecewise smooth  system $(I)$ is a center if and only if $a_1<0, b_1>0$ and $\tilde{a}_1>0$ and furthermore
it is a global center.


\medskip

Next, in the case $a_1<0, b_1>0$ and $\tilde{a}_1>0$ we  prove that the center $O_1$  at the origin of piecewise smooth  system $(I)$ is not  isochronous.
Assuming that  $\Gamma_{r_0}$ is the closed trajectory through $(r_0, 0)$ inside the periodic annulus of the  center $O_1$,
we can define the positive half-period function as
 $
 T^+(r_0):= \int_{\Gamma^+_{r_0}} ~dt
 $
 and  the negative  half-period function as
 $
 T^-(r_0):= \int_{\Gamma^-_{r_0}} ~dt,
 $
 where $r_0>0$,
 $$
 \Gamma^+_{r_0}= \{(\rho,\theta)\in \mathbb{R}^2  ~|~\rho=p^+(r_0, \theta)\}=\{(x,y)\in \mathbb{R}^2  ~|~y= \Big( \frac{3 b_1}{2 a_1} (x^2-r_0^2) \Big)^{\frac{1}{3}} ~\}
$$
and
$$
\Gamma^-_{r_0}= \{(\rho,\theta)\in \mathbb{R}^2  ~|~\rho=p^-(r_0, \theta)\}=\{(x,y)\in \mathbb{R}^2  ~|~y= \Big( \frac{3}{2 \tilde{a}_1} (x^2-r_0^2) \Big)^{\frac{1}{3}} ~\}.
$$
 Thus the complete period function associated to  piecewise smooth  system $(I)$ is given by the sum of these two functions
\begin{equation*}\label{Per-func}
\begin{aligned}
T_I(r_0)&= \oint_{\Gamma_{r_0}} ~dt=T^+(r_0)+T^-(r_0)
=\int_{\Gamma^+_{r_0}} ~\frac{dx}{a_1 y^2} + \int_{\Gamma^-_{r_0}} ~\frac{dx}{\tilde{a}_1 y^2}
\\
& = \int_{r_0}^{-r_0} ~\frac{dx}{a_1  \Big( \frac{3 b_1}{2 a_1} (x^2-r_0^2) \Big)^{\frac{2}{3}}}
+ \int_{-r_0}^{r_0} ~\frac{dx}{\tilde{a}_1  \Big( \frac{3}{2 \tilde{a}_1} (x^2-r_0^2) \Big)^{\frac{2}{3}} }
\\
& = \beta_0 r_0^{-\frac{1}{3}},
\end{aligned}
\end{equation*}
where $\beta_0=\frac{ (\frac{2}{3})^{\frac{5}{3}} \pi^{\frac{3}{2}} \sqrt{3} }{ \Gamma(\frac{2}{3})\Gamma(\frac{5}{6}) }
(-a_1^{-\frac{1}{3}} b_1^{-\frac{2}{3}}+\tilde{a}_1^{-\frac{1}{3}})>0$ 
and the  Gamma function  $\Gamma(z) = \int_0^{\infty} ~e^{-s} s^{z-1} ~ ds $.
Clearly the period $T_I(r_0)$ of the periodic orbits inside the period annulus of the center $O_1$ is monotonic in $r_0$ and then it cannot
be isochronous.
We complete the proof of the theorem.
\end{proof}

\bigskip


\section{Global structures  of piecewise smooth  quadratic quasi--homogeneous  systems}
\label{s3}


We will apply the ideal of Poincar\'e  compactification to study the global structures of piecewise smooth
quadratic quasi--homogeneous but non--homogeneous systems.
Although this theory is usually used  in smooth systems, our strategy is to analyze the properties at infinity on half plane one by one,
and via the discussion of sliding sets and crossing sets
we can  summarize the global topological structures of piecewise smooth  quadratic quasi--homogeneous  systems.

First we briefly remind the procedure of Poincar\'e  compactification   \cite{AndLGM1973, DumoLA2006}.
Consider  a planar vector field
$$
\mathcal{X}= \tilde{P}(x,y) \dfrac{\partial}{\partial x} + \tilde{Q}(x,y) \dfrac{\partial}{\partial y},
$$
where $\tilde{P}(x,y)$ and $\tilde{Q}(x,y)$ are polynomials of degree $n$.
Set  $\mathbb{S}^2 = \{ y=(y_1,y_2,y_3) \in  \mathbb{R}^3 : y_1^2+y_2^2+y_3^2 =1 \}$,
$\mathbb{S}^1$ be the equator of  $\mathbb{S}^2$ and $p(\mathcal{X})$ be the \textit{Poincar\'e compactification} of $\mathcal{X}$ on $\mathbb{S}^2$.
Note that $\mathbb{S}^1$ is invariant under the flow of $p(\mathcal{X})$.

We consider the six local charts $U_i = \{ y \in  \mathbb{S}^2 : y_i > 0 \}$ and $V_i = \{ y \in  \mathbb{S}^2 : y_i < 0 \}$ where $i = 1, 2, 3$
for the calculation of the expression of $p(\mathcal{X})$. 
The diffeomorphisms $F_i : U_i \rightarrow \mathbb{R}^2$ and $G_i : V_i \rightarrow \mathbb{R}^2$ for $i = 1, 2, 3$ are the inverses of the central projections from the planes tangent at the points $(1, 0, 0), (-1, 0, 0), (0, 1, 0), (0,-1, 0), (0, 0, 1)$ and $(0, 0,-1)$ respectively.
We denote by $(u, z)$ the value of $F_i(y)$ or $G_i(y)$ for any $i = 1, 2, 3$.
The expression for $p(\mathcal{X})$ in the local chart $(U_1, F_1)$ is given by
$$
\dot{u} = z^n \left[ -u \tilde{P} \left( \dfrac{1}{z}, \dfrac{u}{z}\right) + \tilde{Q} \left( \dfrac{1}{z}, \dfrac{u}{z}\right) \right],
\hspace{0.3cm}
\dot{z} = - z^{n+1} \tilde{P} \left( \dfrac{1}{z}, \dfrac{u}{z}\right),
$$
for $(U_2, F_2)$ is
$$
\dot{u} = z^n \left[ \tilde{P} \left( \dfrac{u}{z}, \dfrac{1}{z}\right) -u \tilde{Q} \left( \dfrac{u}{z}, \dfrac{1}{z}\right) \right],
\hspace{0.3cm}
\dot{z} = - z^{n+1} \tilde{Q} \left( \dfrac{u}{z}, \dfrac{1}{z}\right),
$$
and for $(U_3, F_3)$ is
$$
\dot{u} = \tilde{P}(u,z),
\hspace{0.3cm}
\dot{z} = \tilde{Q}(u,z).
$$
When we study the equilibria at infinitt on the charts $U_2 \cup V_2$, we only need to verify if the origins of these charts are
singular points.


\medskip


\begin{theorem}
\label{th-X111}
Piecewise smooth  quadratic quasi--homogeneous systems $(I)$, $(II)$ and $(III)$ have totally $8$, $64$ and $36$
global phase portraits  respectively.
\end{theorem}

\begin{proof}
For piecewise smooth  system $(I)$, the origin of the corresponding smooth system  $(i)$  is a cusp by the proof of Theorem \ref{th-center}.
Moreover, in the half plane $y>0$ (resp. $y<0$) the piecewise smooth  system $(I)$ has an invariant curve $\frac{a_1}{3}y^3 -\frac{b_1}{2} x^2=0$
(resp. $\frac{\tilde{a}_1}{3}y^3 -\frac{1}{2} x^2=0$) passing through the origin $O_1:(0,0)$ when $a_1b_1>0$ (resp. $\tilde{a}_1<0$).

Taking respectively the Poincar\'e transformations $x=1/z, \, y=u/z$ in the local chart $U_1$
and $x=u/z, \, y=1/z$ in the local chart   $U_2$, 
smooth system  $(i)$  around the equator of the Poincar\'e sphere can be written respectively  in
\begin{eqnarray}
\dot{u}= b_1 z-a_1 u^3,  \qquad
\dot{z}=-a_1 u^2 z
\label{I-Ix}
\end{eqnarray}
and
\begin{eqnarray*}
\dot{u}=a_1-b_1 u^2 z,  \qquad
\dot{z}=- b_1 u z^2.
\label{I-Iy}
\end{eqnarray*}
Then  singularities at infinity of system $(i)$ only exist on the $x$-axis,
whose corresponding   equilibrium in the local chart $U_1$ (the origin of \eqref{I-Ix}) is a node   by  \cite[Theorem 3.5]{DumoLA2006}.
Besides, system $(i)$ has the polynomial first integral
$H_1^+(x, y)=\frac{a_1}{3}y^3 -\frac{b_1}{2} x^2$. Therefore, it is not difficult to get global phase portraits  of smooth system $(i)$.

Notice that piecewise smooth  system $(I)$ is invariant  under the change $(x, y)\to (-x, y)$ after a time rescaling $d\tau\to -d\tau$,
so we only need to consider phase portraits  in the half plane $x\ge0$.
From \eqref{I-c} and \eqref{I-s}, we know that  the crossing set
$\mathcal{L}_c^I= \{(x,y)\in\mathbb{R}^2 | ~y=0,  x\ne 0\}$ if $b_1>0$ and  the sliding set
$ \mathcal{L}_s^I =\{(x,y)\in  \mathbb{R}^2 | ~y=0\}$ if $b_1<0$. The origin $O_1$ is the unique  singular sliding point of piecewise smooth  system $(I)$.
Hence when $b_1>0$, the global phase portraits of piecewise smooth  system $(I)$ can be obtained by the global phase portraits of  system $(i)$
in the half plane $y>0$ and $y<0$ respectively, 
where four cases 
$b_1>0, a_1>0, \tilde{a}_1>0$;  $b_1>0, a_1>0, \tilde{a}_1<0$; $b_1>0, a_1<0, \tilde{a}_1<0$ and $b_1>0, a_1<0, \tilde{a}_1>0$ are considered.
Notice that the whole plane is the period annulus of the center at the origin  of  system $(I)$ if $b_1>0, a_1<0, \tilde{a}_1>0$.
There exist infinitely many homoclinic loops connecting with the singularities at infinity on the $x$-axis and
an eight-shape heteroclinic loop connecting with the singularities at infinity on the $x$-axis and the origin if  $b_1>0, a_1>0, \tilde{a}_1<0$.
In contrast, when $b_1<0$ the whole $x$-axis is the sliding set and each point except the origin on the $x$-axis is a "colliding" point of orbits,
i.e., the orbit  connecting the point from the half plane $y>0$ is along the opposite direction with that from the half plane  $y<0$.
Thus, neither closed orbits nor  homoclinic loops could exist if  $b_1<0$. For $b_1<0$, we research the global phase portraits of piecewise smooth  system $(I)$
in four cases: $a_1>0, \tilde{a}_1>0$;  $ a_1>0, \tilde{a}_1<0$; $ a_1<0, \tilde{a}_1<0$ and $ a_1<0, \tilde{a}_1>0$.
Note that in the case  $b_1<0, a_1>0, \tilde{a}_1>0$  it seems that the origin is surrounded by closed orbits but it is not true,
since the direction of upper half of each oval is clockwise but the direction of the lower half is anticlockwise.
Thus the ovals existing in the case $b_1<0, a_1<0, \tilde{a}_1>0$ are not  homoclinic loops actually by the similar reason.
Thus we obtain $8$ global phase portraits for piecewise smooth system  $(I)$ under above $8$ parameter conditions.
\medskip

We next investigate the  piecewise smooth  system $(II)$ for its global structures. Using  Theorem 3.5 of  \cite{DumoLA2006},
the origin of  smooth system  $(ii)$   is a saddle  if $a_2b_{22}<0$,
and its neighborhood consists of a hyperbolic sector and an elliptic sector if $a_2b_{22}>0$.
Taking respectively the Poincar\'e transformations
$x=1/z, \, y=u/z$ in the local chart $U_1$ and $x=u/z, \, y=1/z$ in the local chart $U_2$, 
system  $(ii)$  around the equator of the Poincar\'e sphere can be written respectively as
\begin{eqnarray}
\dot{u}= b_{21} z +(b_{22}-a_2) u^2,   \qquad
\dot{z}= -a_2 u z
\label{II-Ix}
\end{eqnarray}
and
\begin{eqnarray}
\dot{u}=  (a_2-b_{22})u - b_{21} u^2 z, \qquad
\dot{z}= -b_{22} z -b_{21} u z^2.
\label{II-Iy}
\end{eqnarray}
Therefore, there exist singularities  of system $(ii)$ located at the infinity of both the $x$-axis and the $y$-axis  if $ b_{22}\ne a_2 $,
which are associated to the origins of system \eqref{II-Ix} and system \eqref{II-Iy} respectively.
It is easy to see that the origin of system  \eqref{II-Iy} is a saddle if $(b_{22}-a_2)b_{22}<0$
and a node if $(b_{22}-a_2)b_{22}>0$. Applying \cite[Theorem 3.5]{DumoLA2006}, the origin of system
\eqref{II-Ix} is a saddle if $(b_{22}-a_2)a_2>0$ and its neighborhood consists of a hyperbolic sector and an elliptic sector if $(b_{22}-a_2)a_2<0$.
 When $ b_{22}=a_2$, the infinity is full up with singularities and there exists a unique orbit connecting with each point at infinity.

From \eqref{II-c} and  \eqref{II-s}, we get that the crossing set
$\mathcal{L}_c^{II}= \{(x,y)\in\mathbb{R}^2 | ~y=0,  x\ne 0\}$ if $b_{21}>0$ and the sliding set
$ \mathcal{L}_s^{II} =\{(x,y)\in  \mathbb{R}^2 | ~y=0\}$ if $b_{21}<0$ for piecewise smooth  system $(II)$.
The origin $O_2:(0,0)$ is the unique  singular sliding point of piecewise smooth  system $(II)$.
Moreover, system $(II)$ has a first integral
\begin{eqnarray*}
H_2^+(x, y) =  \frac{-2 b_{21} x+(a_2-2 b_{22}) y^2}{x^{\frac{2 b_{22}}{a_2}} (-2 b_{22}+a_2)}
\qquad
(resp.  = \frac{- b_{21} x \ln x+b_{22}y^2}{ b_{22}x})
\end{eqnarray*}
when $y\ge 0$ and $a_2\ne 2 b_{22}$ (resp. $a_2=2 b_{22}$), and a first integral
\begin{eqnarray*}
H_2^-(x, y) =  \frac{-2  x+(\tilde{a}_2-2) y^2}{x^{\frac{2}{\tilde{a}_2}} (-2 +\tilde{a}_2)}
\qquad
(resp.  = \frac{x \ln x-y^2}{x })
\end{eqnarray*}
 when $y<0$ and $\tilde{a}_2\ne 2  $ (resp. $\tilde{a}_2=2 $).
Similar to the research of system $(I)$,    the global phase portraits of piecewise smooth  system $(II)$ can be obtained by the global phase portraits of  system $(ii)$
in the half planes $y>0$ and $y<0$ together with the dynamics on the crossing set and sliding set, 
where $64$ subcases   correspond to parameter conditions obtained by the signs of
$b_{21}$, $b_{22}-a_2$, $b_{22}$, $a_2$, $\tilde{a}_2$ and $\tilde{a}_2-1$.

\smallskip

In order to research the global dynamics of piecewise smooth   system $(III)$, we need find global dynamics of  smooth system  $(iii)$.
 Obviously, the origin $O_3:(0,0)$ of  system $(iii)$ is a saddle if $a_{31}b_3<0$ and a node if $a_{31}b_3>0$.

In the local charts $U_1$ and $U_2$ of the Poincar\'e sphere,  system  $(iii)$ becomes
\begin{eqnarray}
\dot{u}=    (b_3-a_{31}) uz-a_{32} u^3,  \qquad
\dot{z}=  -a_{31} z^2-a_{32} u^2z
\label{III-Ix}
\end{eqnarray}
and
\begin{eqnarray*}
\dot{u}= a_{32}+(a_{31} -b_3) u z,  \qquad
\dot{z}=-b_3 z^2,
\label{III-Iy}
\end{eqnarray*}
respectively. 
Then there exist singularities at  infinity of system $(iii)$ only located on the $x$-axis,
which are associated to the origin of \eqref{III-Ix}, a high degenerate equilibrium.
More precisely, the neighborhood of the origin of \eqref{III-Ix} consists of two elliptic sectors and one parabolic
sector if $a_{31} b_3<0$,   two hyperbolic sectors and two parabolic sectors
if $0< a_{31} b_3\le 2b_3^2$, and   two hyperbolic sectors and four parabolic sectors if  $a_{31} b_3>2b_3^2$
by applying results of  Reyn \cite[Figures 8.3c-8.3d]{Reyn2007}. 

The  sliding set of piecewise smooth system $(III)$ is the whole $x$-axis from \eqref{III-cs}, which is filled with singular sliding points.
Except the origin and singularities  of system $(III)$ located at the infinity of the $x$-axis,
no orbits  connect with the point in the $x$-axis from the half planes $y>0$ or  $y<0$.
Hence, neither closed orbits nor sliding closed orbits  could exist. There exist no homoclinic loops in a bounded region, that is,
a homoclinic loop  has to pass by a singularity at infinity of the $x$-axis if it exists.
Besides, system $(III)$ has a first integral
\begin{eqnarray*}
H_3^+(x, y) =    \frac{(a_{31}-2 b_3) x +a_{32} y^2}{y^{ \frac{a_{31}}{b_3}} (-2 b_3+a_{31})}
\qquad
(resp.  =  - \frac{a_{32} y^2\ln y-b_3 x}{b_3 y^2}  )
\end{eqnarray*}
when $y\ge 0$ and $a_{31}\ne2 b_3$ (resp. $a_{31}=2 b_3$), and a first integral
\begin{eqnarray*}
H_3^-(x, y) =  \frac{(\tilde{a}_{31}-2) x +y^2}{y^{ \tilde{a}_{31}} (-2+\tilde{a}_{31})},
\qquad
(resp.  =  - \frac{y^2\ln y-x}{y^2} )
 \end{eqnarray*}
 when $y< 0$ and $\tilde{a}_{31}\ne2$ (resp. $\tilde{a}_{31}=2$).
From an analogous discussion of  system $(I)$ in the case $b_1<0$,  the above analysis of system $(III)$ provides enough preparation for studying
the global structure of piecewise smooth  system $(III)$ and  we have its 36 global phase portraits
 by  the signs of  $a_{31}$, $ b_3$, $a_{31}-2b_3$, $a_{32}$,  $\tilde{a}_{31}$ and  $\tilde{a}_{31}-2$ respectively.

 Summarizing the above investigation, we can obtain  global dynamics of piecewise smooth  quadratic quasi--homogeneous systems $(I)$-$(III)$.
The proof  is completed.
\end{proof}

Here for simplicity, we only present details of topological phase portraits of piecewise  smooth system $(III)$
 and omit that of  systems $(I)$ and $(II)$.
The parameter conditions associated  to cases $(1)$-$(36)$ are give in Table 1.
Remark that  we will not consider invertible changes which transform the half plane $y>0$ into the half plane $y<0$ for the topological equivalence of global phase portraits,  in the sense that the vector fields of piecewise  smooth systems are different in the half planes $y>0$ and $y<0$.

Note that for smooth quadratic  quasi--homogeneous system $(i)$ (resp. $(ii)$, $(iii)$)  there only exists $1$ (resp. $4$, $3$)
 global   phase portrait  without taking into account the direction of the time,
but piecewise smooth  quadratic  quasi--homogeneous system $(I)$ (resp. $(II)$, $(III)$)
has  $8$ (resp. $64$, $36$) global  phase portraits.
Thus piecewise smooth   quadratic  quasi--homogeneous systems can exhibit
more complicated and richer dynamics than the smooth ones.

%
%
%
%
%
%
%

 \smallskip


 \begin{center}
 {\footnotesize
\begin{tabular}{|c|c|c|}
\hline
cases & Parameter conditions
\\
\hline
 (1) &  $a_{31}<0$, $ b_3<0$, $a_{31}\ge 2b_3$, $a_{32}<0$ and  $\tilde{a}_{31}>2$
\\
\hline
(2) &  $a_{31}<0$, $ b_3<0$, $a_{31}\ge 2b_3$,  $a_{32}<0$ and  $0<\tilde{a}_{31}\le 2$
\\
\hline
(3) &  $a_{31}<0$, $ b_3<0$, $a_{31}\ge 2b_3$,  $a_{32}<0$ and  $ \tilde{a}_{31}<0$
\\
\hline
(4) &  $a_{31}<0$, $ b_3<0$, $a_{31}\ge 2b_3$,  $a_{32}>0$ and  $ \tilde{a}_{31}>2$
\\
\hline
(5) &  $a_{31}<0$, $ b_3<0$, $a_{31}\ge 2b_3$,  $a_{32}>0$ and  $0<\tilde{a}_{31}\le 2$
\\
\hline
(6) &  $a_{31}<0$, $ b_3<0$, $a_{31}\ge 2b_3$,  $a_{32}>0$ and  $ \tilde{a}_{31}<0$
\\
\hline
(7) &  $a_{31}<0$, $ b_3<0$, $a_{31}< 2b_3$,  $a_{32}>0$ and  $ \tilde{a}_{31}>2$
\\
\hline
(8) &  $a_{31}<0$, $ b_3<0$, $a_{31}< 2b_3$,  $a_{32}>0$ and  $0<\tilde{a}_{31}\le 2$
\\
\hline
(9) &  $a_{31}<0$, $ b_3<0$, $a_{31}< 2b_3$,  $a_{32}>0$ and  $ \tilde{a}_{31}<0$
\\
\hline
(10) &  $a_{31}<0$, $ b_3<0$, $a_{31}< 2b_3$,  $a_{32}<0$ and  $ \tilde{a}_{31}>2$
\\
\hline
(11) &  $a_{31}<0$, $ b_3<0$, $a_{31}< 2b_3$,  $a_{32}<0$ and  $0<\tilde{a}_{31}\le 2$
\\
\hline
(12) &  $a_{31}<0$, $ b_3<0$, $a_{31}< 2b_3$,  $a_{32}<0$ and  $ \tilde{a}_{31}<0$
\\
\hline
(13) &  $a_{31}>0$, $ b_3>0$, $a_{31}\le 2b_3$,  $a_{32}<0$ and  $ \tilde{a}_{31}>2$
\\
\hline
(14) &  $a_{31}>0$, $ b_3>0$, $a_{31}\le 2b_3$, $a_{32}<0$ and  $0<\tilde{a}_{31}\le 2$
\\
\hline
(15) & $a_{31}>0$, $ b_3>0$, $a_{31}\le 2b_3$,  $a_{32}<0$ and  $ \tilde{a}_{31}<0$
\\
\hline
(16) &  $a_{31}>0$, $ b_3>0$, $a_{31}\le 2b_3$,  $a_{32}>0$ and  $ \tilde{a}_{31}>2$
\\
\hline
(17) &  $a_{31}>0$, $ b_3>0$, $a_{31}\le 2b_3$, $a_{32}>0$ and  $0<\tilde{a}_{31}\le 2$
\\
\hline
(18) & $a_{31}>0$, $ b_3>0$, $a_{31}\le 2b_3$,  $a_{32}>0$ and  $ \tilde{a}_{31}<0$
\\
\hline
(19) &  $a_{31}>0$, $ b_3>0$, $a_{31}> 2b_3$,  $a_{32}<0$ and  $ \tilde{a}_{31}>2$
\\
\hline
(20) &  $a_{31}>0$, $ b_3>0$, $a_{31}> 2b_3$, $a_{32}<0$ and  $0<\tilde{a}_{31}\le 2$
\\
\hline
(21) & $a_{31}>0$, $ b_3>0$, $a_{31}> 2b_3$,  $a_{32}<0$ and  $ \tilde{a}_{31}<0$
\\
\hline
(22) &  $a_{31}>0$, $ b_3>0$, $a_{31}> 2b_3$,  $a_{32}>0$ and  $ \tilde{a}_{31}>2$
\\
\hline
(23) &  $a_{31}>0$, $ b_3>0$, $a_{31}> 2b_3$, $a_{32}>0$ and  $0<\tilde{a}_{31}\le 2$
\\
\hline
(24) & $a_{31}>0$, $ b_3>0$, $a_{31}> 2b_3$,  $a_{32}>0$ and  $ \tilde{a}_{31}<0$
\\
\hline
(25) & $a_{31}>0$, $ b_3<0$,  $a_{32}<0$ and  $ \tilde{a}_{31}>2$
\\
\hline
(26) &  $a_{31}>0$, $ b_3<0$,  $a_{32}<0$ and  $0<\tilde{a}_{31}\le 2$
\\
\hline
(27) & $a_{31}>0$, $ b_3<0$,  $a_{32}<0$ and  $ \tilde{a}_{31}<0$
\\
\hline
(28) & $a_{31}>0$, $ b_3<0$,  $a_{32}>0$ and  $ \tilde{a}_{31}>2$
\\
\hline
(29) &  $a_{31}>0$, $ b_3<0$,  $a_{32}>0$ and  $0<\tilde{a}_{31}\le 2$
\\
\hline
(30) & $a_{31}>0$, $ b_3<0$,  $a_{32}>0$ and  $ \tilde{a}_{31}<0$
\\
\hline
(31) & $a_{31}<0$, $ b_3>0$,  $a_{32}<0$ and  $ \tilde{a}_{31}>2$
\\
\hline
(32) &  $a_{31}<0$, $ b_3>0$,  $a_{32}<0$ and  $0<\tilde{a}_{31}\le 2$
\\
\hline
(33) & $a_{31}<0$, $ b_3>0$,  $a_{32}<0$ and  $ \tilde{a}_{31}<0$
\\
\hline
(34) & $a_{31}<0$, $ b_3>0$,  $a_{32}>0$ and  $ \tilde{a}_{31}>2$
\\
\hline
(35) &  $a_{31}<0$, $ b_3>0$,  $a_{32}>0$ and  $0<\tilde{a}_{31}\le 2$
\\
\hline
(36) & $a_{31}<0$, $ b_3>0$,  $a_{32}>0$ and  $ \tilde{a}_{31}<0$
\\
\hline
\end{tabular}
}
\vskip 0.2cm

Table 1. Parameter conditions for global  phase portraits of  system $(III)$. 
\end{center}

From the global dynamics of piecewise  smooth systems $(I)$-$(III)$, we can research the global bifurcation and  unfolding  of all special orbits
including homoclinic loops (or heteroclinic loops), closed orbits, equilibria and equilibria at infinity for the systems.
For example,  the unfoldings are presented in phase portraits (2) and (3)  if we choose $\mu_1=\tilde{a}_{31}$ as an unfolding parameter.
A  homoclinic (or heteroclinic) bifurcation happens when $\mu_1$ passes through zero. More precisely, when $\mu_1>0$
there exist infinitely many heteroclinic loops connecting with the origin and equilibria at infinity on the $x$-axis, and
 when $\mu_1<0$ some heteroclinic orbits in the half plane $y<0$ become homoclinic loops connecting with the equilibria at infinity on the $x$-axis.
 If we choose $\mu_2=a_{32}$ as another unfolding parameter, we find that the heteroclinic loops burst out when $\mu_2$ varies from negative to positive
 by phase portraits  (2) and (5). At the same time it can be observed clearly the change of sectors in a neighborhood of equilibria and equilibria at infinity, which  exhibits  bifurcations of equilibria.
 We can also notice other global and local bifurcations if we choose  $ b_3$, $a_{31}-2b_3$,   $a_{31}$ and  $\tilde{a}_{31}-2$ as unfolding parameters
 for piecewise smooth system $(III)$. 

\bigskip


\section{Limit cycle  bifurcations  by perturbing piecewise smooth  quadratic quasi--homogeneous  systems}

From Theorem \ref{th-center},  only system  $(I)$ of all piecewise smooth   quadratic quasi-homogeneous  systems has a center at the origin,
which is global if it exists. In this section we research the bifurcation of limit cycles by perturbing
piecewise smooth  quadratic quasi--homogeneous  system  $(I)$ with arbitrary piecewise polynomials of degree $n\in \mathbb{N}$, where $\mathbb{N}=\mathbb{Z}_+\cup \{0\}$.

Consider the following one-parametric family of piecewise smooth systems
\begin{eqnarray}
     \dot{x}= H_y(x,y) +  \epsilon f(x,y), \qquad
    \dot{y}=-H_x(x,y) + \epsilon g(x,y),
 \label{I-epsilon}
\end{eqnarray}
where $\epsilon \in \mathbb{R}$ is the small perturbation parameter,
\begin{equation*}
 H(x,y)=
 \begin{cases}
H^+(x, y):=H_1^+(x, y)=\frac{a_1}{3}y^3 -\frac{b_1}{2} x^2,    ~\mbox{ if }~ y\ge 0,
\\
H^-(x, y):=H_1^-(x, y)=\frac{\tilde{a}_1}{3}y^3 -\frac{1}{2} x^2,  ~ \mbox{ if }~ y < 0,
 \end{cases}
\end{equation*}
\begin{equation*}
 f(x,y)=
 \begin{cases}
f^+(x, y)= \sum_{i+j=0}^n c_{ij}^+ ~ x^i y^j,    ~\mbox{ if }~ y\ge 0,
\\
f^-(x, y)= \sum_{i+j=0}^n c_{ij}^- ~ x^i y^j,  ~ \mbox{ if }~ y < 0,
 \end{cases}
\end{equation*}
\begin{equation*}
 g(x,y)=
 \begin{cases}
g^+(x, y)= \sum_{i+j=0}^n d_{ij}^+ ~ x^i y^j,    ~\mbox{ if }~ y\ge 0,
\\
g^-(x, y)= \sum_{i+j=0}^n d_{ij}^- ~ x^i y^j,  ~ \mbox{ if }~ y < 0
 \end{cases}
\end{equation*}
for arbitrary  $ c_{ij}^{\pm}, d_{ij}^{\pm}\in \mathbb{R}$, and $a_1<0, b_1>0$ and $\tilde{a}_1>0$.
Our aim is to give the maximum number of limit cycles in terms of $n$ which can bifurcate from the periodic orbits of the center
 at the origin of system $(I)$  with  $\epsilon=0$, inside the family \eqref{I-epsilon}  for nonzero $\epsilon$.

We will use Melnikov method to investigate the number of bifurcated limit cycles from system \eqref{I-epsilon}.
Let
\begin{align*}
 L_h^+ = \{(x,y)\in \mathbb{R}^2 \mid H_1^+(x, y)= -\frac{h}{2}, ~h>0 \},   ~\mbox{ if }~ y\ge 0,
\\
 L_h^- = \{(x,y)\in \mathbb{R}^2 \mid H_1^-(x, y)= -\frac{h}{2b_1}, ~h>0 \},  ~ \mbox{ if }~ y < 0.
 \end{align*}
Then the  family of periodic orbits of system  \eqref{I-epsilon} with  $\epsilon=0$ is presented by $L_h=L_h^+ \cup L_h^-$, where $h>0$.

Using the idea in \cite{LiuH2010} for piecewise smooth  system with a discontinuous line the $y$-axis,
we have the first order Melnikov function for system \eqref{I-epsilon} along the family of periodic orbits $L_h$,
which is
\begin{eqnarray}
\qquad\qquad M(h, \epsilon)=\frac{H_x^+(A)}{H_x^-(A)}  \Big(\frac{H_x^-(B)}{H_x^+(B)} \int_{L_h^+} (g^+dx-f^+dy) +\int_{L_h^-} (g^-dx-f^-dy) \Big),
 \label{Melikov}
\end{eqnarray}
where points $A=(\sqrt{\frac{h}{b_1}}, 0)$ and $B=(-\sqrt{\frac{h}{b_1}}, 0)$,  as shown in Figure \ref{Fig-Mel}.

\begin{figure}[htb]
\centering
\includegraphics[scale=.70]{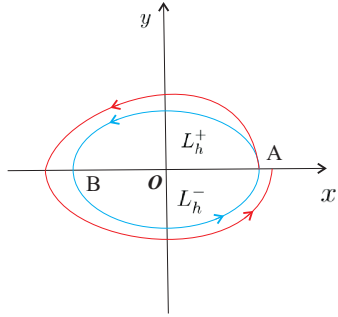}
\caption{ The closed orbit of system  $(I)$ and its perturbation.  }
\label{Fig-Mel}
\end{figure}


\begin{lemma}
\label{lm-Mel}
For  piecewise smooth  system \eqref{I-epsilon}, 
we have the  first order  Melnikov function
\begin{eqnarray}
M(h, \epsilon)=    \sum_{2k+j=0}^n  \xi_{2k,j} ~h^{k+\frac{j}{3}+\frac{1}{2}},
 \label{Melikov2}
\end{eqnarray}
where coefficients $ \xi_{2k,j}$ are given in \eqref{xi}. 
\end{lemma}

\begin{proof}

Firstly, we compute that $H_x^+(A)=-b_1 \sqrt{\frac{h}{b_1}}$, $H_x^-(A)= -\sqrt{\frac{h}{b_1}}$,
$H_x^-(B)=\sqrt{\frac{h}{b_1}}$ and $H_x^+(B)=b_1\sqrt{\frac{h}{b_1}}$ in \eqref{Melikov}.

 Restricted on $L_h^+$ and $L_h^-$, we solve $y= \sqrt[3]{\frac{3}{2a_1}(-h+ b_1 x^2)}$
 and   $y=\sqrt[3]{ \frac{3}{2\tilde{a}_1}(-\frac{h}{b_1}+ x^2)}$, respectively.
 Then for $i, j \in \mathbb{N}$ we calculate
 \begin{eqnarray}
  \label{dx+}
  \begin{array}{lll}
  \int_{L_h^+}    ~ x^i y^j dx
    & =  (\frac{3}{2a_1})^{\frac{j}{3}} \int_{\sqrt{\frac{h}{b_1}}}^{-\sqrt{\frac{h}{b_1}}}    x^i (-h+ b_1 x^2)^{\frac{j}{3}} dx
  \\
 & = -2  (\frac{3}{2a_1})^{\frac{j}{3}} \int_0^{\sqrt{\frac{h}{b_1}}}    x^{2k} (-h+ b_1 x^2)^{\frac{j}{3}}    dx
 \\
 &= \hat{d}_{2k,j}^+  ~h^{k+\frac{j}{3}+\frac{1}{2}},
 \end{array}
\end{eqnarray}
\begin{eqnarray}
  \label{dy+}
  \begin{array}{lll}
  \int_{L_h^+}    ~ x^i y^j dy &= \frac{b_1}{a_1} \int_{L_h^+}    ~ x^{i+1} y^{j-2} dx
  \\
    & =  \frac{b_1}{a_1} (\frac{3}{2a_1})^{\frac{j-2}{3}} \int_{\sqrt{\frac{h}{b_1}}}^{-\sqrt{\frac{h}{b_1}}}    x^{i+1} (-h+ b_1 x^2)^{\frac{j-2}{3}} dx
  \\
 & = -2  \frac{b_1}{a_1} (\frac{3}{2a_1})^{\frac{j-2}{3}} \int_0^{\sqrt{\frac{h}{b_1}}}    x^{2k+2} (-h+ b_1 x^2)^{\frac{j-2}{3}}    dx
\\
 &= \hat{c}_{2k+1,j}^+  ~h^{k+\frac{j+1}{3}+\frac{1}{2}},
 \end{array}
\end{eqnarray}
\begin{eqnarray}
  \label{dx-}
  \begin{array}{lll}
  \int_{L_h^-}    ~ x^i y^j dx
    & =  (\frac{3}{2\tilde{a}_1})^{\frac{j}{3}} \int_{-\sqrt{\frac{h}{b_1}}}^{\sqrt{\frac{h}{b_1}}}    x^i (x^2-\frac{h}{b_1})^{\frac{j}{3}} dx
  \\
 & = 2  (\frac{3}{2\tilde{a}_1})^{\frac{j}{3}} \int_0^{\sqrt{\frac{h}{b_1}}}    x^{2k} (x^2-\frac{h}{b_1})^{\frac{j}{3}}    dx
 \\
 &= \hat{d}_{2k,j}^-  ~h^{k+\frac{j}{3}+\frac{1}{2}}
 \end{array}
\end{eqnarray}
and
\begin{eqnarray}
  \label{dy-}
  \begin{array}{lll}
  \int_{L_h^-}    ~ x^i y^j dy &= \frac{1}{\tilde{a}_1} \int_{L_h^-}    ~ x^{i+1} y^{j-2} dx
  \\
   & =  \frac{1}{\tilde{a}_1} (\frac{3}{2\tilde{a}_1})^{\frac{j-2}{3}} \int_{-\sqrt{\frac{h}{b_1}}}^{\sqrt{\frac{h}{b_1}}}
   x^{i+1} (x^2-\frac{h}{b_1})^{\frac{j-2}{3}} dx
  \\
 & = 2  \frac{1}{\tilde{a}_1} (\frac{3}{2\tilde{a}_1})^{\frac{j-2}{3}} \int_0^{\sqrt{\frac{h}{b_1}}}    x^{2k+2} (x^2-\frac{h}{b_1})^{\frac{j-2}{3}}    dx
  \\
 &= \hat{c}_{2k+1,j}^-  ~h^{k+\frac{j+1}{3}+\frac{1}{2}}
 \end{array}
\end{eqnarray}
where $k\in \mathbb{N}$,
\begin{eqnarray}
  \label{hat-cd}
  \begin{array}{lll}
  \hat{d}_{2k,j}^+  &=
   -2  (\frac{3}{2a_1})^{\frac{j}{3}} (\frac{1}{b_1})^{k+\frac{1}{2}}  \int_0^1    x^{2k} (-1+ x^2)^{\frac{j}{3}}    dx,
   \\
     \hat{c}_{2k+1,j}^+
   & = -2  \frac{b_1}{a_1}  (\frac{3}{2a_1})^{\frac{j-2}{3}} (\frac{1}{b_1})^{k+1+\frac{1}{2}}  \int_0^1    x^{2k+2} (-1+ x^2)^{\frac{j-2}{3}}    dx,
 \\
  \hat{d}_{2k,j}^-  & =  2  (\frac{3}{2\tilde{a}_1})^{\frac{j}{3}} (\frac{1}{b_1})^{k+\frac{j}{3}+\frac{1}{2}}  \int_0^1    x^{2k} (-1+ x^2)^{\frac{j}{3}}    dx,
 \\
   \hat{c}_{2k+1,j}^-
  & =  2  \frac{1}{\tilde{a}_1}  (\frac{3}{2\tilde{a}_1})^{\frac{j-2}{3}} (\frac{1}{b_1})^{k+\frac{j+1}{3}+\frac{1}{2}}
  \int_0^1    x^{2k+2} (-1+ x^2)^{\frac{j-2}{3}}    dx.
 \end{array}
\end{eqnarray}
Notice that $\hat{d}_{2k,j}^+  \hat{c}_{2k+1,j}^+ \hat{d}_{2k,j}^- \hat{c}_{2k+1,j}^- \ne 0$,  $\int_{L_h^{\pm}}    ~ x^i y^j dx=0$ for odd $i$ and  $\int_{L_h^{\pm}}    ~ x^i y^j dy=0$ for even $i$.

Substituting \eqref{dx+}--\eqref{dy-} into the formula \eqref{Melikov}, we obtain the
Melnikov function   of system \eqref{I-epsilon} as

 \begin{align*}
 \label{Mel+}
 M(h, \epsilon) & =b_1  \Big(\frac{1}{b_1} \int_{L_h^+} (g^+dx-f^+dy) +\int_{L_h^-} (g^-dx-f^-dy) \Big),
     \\
     & = \int_{L_h^+} \Big(\sum_{i+j=0}^n d_{ij}^+ ~ x^i y^j dx-\sum_{i+j=0}^n c_{ij}^+ ~ x^i y^jdy \Big)
     \\
     & \qquad  +b_1 \int_{L_h^-} \Big(\sum_{i+j=0}^n d_{ij}^- ~ x^i y^j dx-\sum_{i+j=0}^n c_{ij}^- ~ x^i y^jdy \Big)
 \nonumber
 \\
  & =   h^{\frac{1}{2}} \Big(\sum_{2k+j=0}^n d_{2k, j}^+ ~   \hat{d}_{2k,j}^+  ~h^{k+\frac{j}{3}}
  -\sum_{2k+1+j=0}^n c_{2k+1, j}^+ ~ \hat{c}_{2k+1,j}^+  ~h^{k+\frac{j+1}{3}} \Big)
   \\
    & \qquad  +b_1  h^{\frac{1}{2}}  \Big(\sum_{2k+j=0}^n d_{2k, j}^- ~ \hat{d}_{2k,j}^-  ~h^{k+\frac{j}{3}}
    -\sum_{2k+1+j=0}^n c_{2k+1,j}^- ~\hat{c}_{2k+1,j}^-  ~h^{k+\frac{j+1}{3}} \Big)
    \nonumber
 \\
   & =   h^{\frac{1}{2}} \Big(\sum_{2k+j=0}^n d_{2k, j}^+ ~   \hat{d}_{2k,j}^+  ~h^{k+\frac{j}{3}}
  -\sum_{2k+j=0}^n c_{2k+1, j-1}^+ ~ \hat{c}_{2k+1,j-1}^+  ~h^{k+\frac{j}{3}} \Big)
   \\
    & \qquad  +b_1  h^{\frac{1}{2}}  \Big(\sum_{2k+j=0}^n d_{2k, j}^- ~ \hat{d}_{2k,j}^-  ~h^{k+\frac{j}{3}}
    -\sum_{2k+j=0}^n c_{2k+1,j-1}^- ~\hat{c}_{2k+1,j-1}^-  ~h^{k+\frac{j}{3}} \Big)
    \nonumber
   \\
   &= h^{\frac{1}{2}} \sum_{2k+j=0}^n  \xi_{2k,j} ~h^{k+\frac{j}{3}},
 \end{align*}
where $k\in \mathbb{N} $,
 \begin{eqnarray}
  \label{xi}
  \begin{array}{lll}
  \xi_{2k,j}  &= &    d_{2k, j}^+ ~   \hat{d}_{2k,j}^+
  -  c_{2k+1, j-1}^+ ~ \hat{c}_{2k+1,j-1}^+
  \\
  & &~+b_1   d_{2k, j}^- ~ \hat{d}_{2k,j}^-
    -b_1 c_{2k+1,j-1}^- ~\hat{c}_{2k+1,j-1}^-
     \end{array}
\end{eqnarray}
and  $\hat{d}_{2k,j}^+,  \hat{c}_{2k+1,j}^+, \hat{d}_{2k,j}^-, \hat{c}_{2k+1,j}^- $ are displayed in \eqref{hat-cd}.
Therefore, \eqref{Melikov2} is proved.
\end{proof}


In order to determine how many limit cycles  the piecewise smooth  system \eqref{I-epsilon} can have, we analyze
zeros of  Melnikov function  \eqref{Melikov2}. For convenience, we set $h=\hat{h}^6$ and get from  \eqref{Melikov2}
that
\begin{equation}
 \label{Melikov3}
 M(\hat{h}, \epsilon) =   \hat{h}^3 \sum_{2k+j=0}^n  \xi_{2k,j} ~\hat{h}^{6k+2j}
  =\hat{h}^3 \sum_{i+j=0}^n  \xi_{i,j} ~\hat{h}^{3i+2j},
\end{equation}
where $i$ is even.

The zero problem of $M(h, \epsilon)$ is transferred to determine
the cardinal of the set
 \begin{equation*}
 \label{N0}
\mathcal{S}(n)=\{3i+2j: ~ 0\le i+j\le n,   ~i ~~\mbox{even}, ~i, j \in\mathbb{N}  \}.
\end{equation*}
That is, we need to find how many different elements exist in the set
$ \mathcal{S}(n)$.

We denote a trapezoid by
$$
\Upsilon_1(n)=\{(i, j): ~ 0\le i+j\le n,  ~j<3, ~i ~~\mbox{even}, ~i, j \in\mathbb{N} \},
$$
and a set by
$$
\Upsilon(n)=\{3i+2j: ~ 0\le i+j\le n,  ~j<3, ~i ~~\mbox{even}, ~i, j \in\mathbb{N} \}.
$$
The cardinal of the set $ \mathcal{S}(n)$ is given in the following lemma by applying a similar ideal in \cite{GinGL2015}
for  smooth quasi--homogeneous polynomial differential systems.


\begin{lemma}
\label{lm-n}
$\mathcal{S}(n)= \Upsilon(n)$.
\end{lemma}

\begin{proof}
Obviously we have $\mathcal{S}(n) \supset \Upsilon(n)$. We only need to prove that $\mathcal{S}(n) \subset \Upsilon(n)$.

Let $3i+2j$ be an arbitrary element in $\mathcal{S}(n)$ such that $j\ge3$ and $i$ is even.  Then there exists $\iota\in \mathbb{Z}_+$
satisfying $3 \iota \le j< 3(\iota+1)$. We have
$$
3i+2j=3(i+2\iota)+2(j-3\iota)=3i_1+2j_1,
$$
where even $i_1=i+2\iota\ge0$, $0\le j_1=j-3\iota<3$ and $0\le i_1+j_1=i+j-\iota\le n$, implying  $3i_1+2j_1\in \Upsilon(n)$ and $3i+2j \in \Upsilon(n)$.
Hence  $\mathcal{S}(n) \subset \Upsilon(n)$.
 \end{proof}

 Remark that from Lemma \ref{lm-n}   all the values of $3i+2j$ for the degrees of $\hat{h}$ in \eqref{Melikov3}
  are taken exactly   by  points  on the trapezoid  $\Upsilon_1(n)$.

  \medskip


We need use the following version of the Descartes Theorem
proved in \cite{Berezin} to judge  real zeros of the  Melnikov function.

\begin{theorem}[Descartes theorem]\label{Dth}
Consider the real polynomial $q(x) = a_{i_1} x^{i_1} + a_{i_2}
x^{i_2} + \ldots +a_{i_r} x^{i_r}$ with $0 = i_1 < i_2 < \ldots <
i_r$. If $a_{i_j}a_{i_{j+1}} < 0,$ we say that we have a variation
of sign. If the number of variations of signs is $m,$ then the
polynomial $q(x)$ has at most $m$ positive real roots. Furthermore,
always we can choose the coefficients of the polynomial $q(x)$ in
such a way that $q(x)$ has exactly $r-1$ positive real roots.
\end{theorem}

Let $\Xi(n)$ denote the maximal number of limit cycles, which  are produced in  piecewise smooth
system  \eqref{I-epsilon} and  bifurcated from the period solutions of  piecewise smooth  quadratic quasi--homogeneous  system $(I)$
by taking into account the zeros of the first order Melnikov function.

\begin{theorem}
\label{th-cycle}
For piecewise smooth quadratic quasi--homogeneous  system $(I)$ perturbed inside the class of all piecewise smooth  polynomial differential
systems of degree  $n$ when $a_1<0, b_1>0$ and $\tilde{a}_1>0$,  the  number  $\Xi(n)=2[\frac{n+1}{2}]+[\frac{n-1}{2}]-1$
(resp.  $\Xi(n)=2[\frac{n}{2}]+[\frac{n+2}{2}]-1$)
 if $n$ is odd (resp. even) by the first order Melnikov function.
 Moreover, there exist   perturbations of  piecewise smooth  polynomial
systems of degree  $n$ in \eqref{I-epsilon} with exactly $\Xi(n)$  limit cycles.
\end{theorem}

\begin{proof}
From the expression of the first order Melnikov function \eqref{Melikov3} and   Theorem \ref{Dth},
we get that $\Xi(n)$ is equal to   $|\mathcal{S}(n)|-1$, where $|\mathcal{S}(n)|$ is the cardinal of the set $\mathcal{S}(n)$.
Applying Lemma  \ref{lm-n}  we have  $\Xi(n)=|\Upsilon(n)|-1$.

Besides, all the values of the function $3i+2j$ are different for different points on the trapezoid $\Upsilon_1(n)$.
In fact, if $3i+2j=3\tilde{i}+2\tilde{j}$ for both $(i,j)$ and $(\tilde{i},\tilde{j})$ in $\Upsilon_1(n)$, we have
$3(i-\tilde{i})= 2(\tilde{j}-j)$, yielding that $3|(\tilde{j}-j)$. Because $0\le\tilde{j},  j<3$, we get $\tilde{j}-j=0$
and furthermore  $i=\tilde{i}$.
Therefore, each of all   values of the set $\Upsilon(n)$ is taken exactly once by one point on the trapezoid $\Upsilon_1(n)$.

Taking $j=0, 1, 2$ respectively, we calculate
$$
|\Upsilon(n)|= [\frac{n+1}{2}]+[\frac{n+1}{2}]+[\frac{n-1}{2}] = 2[\frac{n+1}{2}]+[\frac{n-1}{2}]
$$
if  $n$ is odd
and
$$
|\Upsilon(n)|= [\frac{n+2}{2}]+[\frac{n}{2}]+[\frac{n}{2}] =2[\frac{n}{2}]+[\frac{n+2}{2}]
$$
 if $n$ is even. Thus, the formula of $\Xi(n)$ in this theorem is proved.

 In addition, we notice that  the Melnikov function \eqref{Melikov3} has $\Lambda:=|\Upsilon(n)|$ terms with different degrees of $\hat{h}$,
 whose coefficients are denoted by
 $$
 \xi_1, ~\xi_2, ...,~ \xi_{\Lambda}.
 $$
These coefficients are linear combinations of parameters
$d_{2k,j}^+,  c_{2k+1,j-1}^+, d_{2k,j}^-$ and $c_{2k+1,j-1}^- $, as seen in \eqref{xi}.
It reveals that   the matrix
$$
\frac{\partial(\xi_1, ~\xi_2, ...,~ \xi_{\Lambda})}
{\partial( c_{1,0}^+,c_{3,0}^+,..., c_{2k+1,j-1}^+,..., c_{2k+1,j-1}^-,..., d_{2k,j}^+,...,d_{2k,j}^-,...)}
$$
 has a full row rank. So there exists an array
 $$
 (c_{1,0}^+,c_{3,0}^+,..., c_{2k+1,j-1}^+, ..., c_{2k+1,j-1}^-,..., d_{2k,j}^+,...,d_{2k,j}^-,...)
 $$
 satisfying that the Melnikov function  $M(\hat{h}, \epsilon)$  in \eqref{Melikov3} has  $\Lambda-1$ variations of signs. By Theorem \ref{Dth},
 the   function $M(\hat{h}, \epsilon)$  has exactly $\Xi(n)$ positive zeros.
 Then we obtain that the Melnikov function  in \eqref{Melikov2}  has exactly $\Xi(n)$ positive zeros and
  piecewise smooth  system \eqref{I-epsilon}  has $\Xi(n)$  limit cycles bifurcating from the periodic solutions of
the piecewise smooth  quadratic quasi--homogeneous center of  system $(I)$    by using the first order Melnikov function.
 \end{proof}


\section*{Acknowledgements}

The  author has received funding from the European Union's Horizon 2020 research and innovation
programme under the Marie Sklodowska-Curie grant agreement No 655212, and is  partially supported by the National Natural
Science Foundation of China (No. 11431008).
The  author thanks Professor Valery   Romanovski for fruitful discussions on the work.



\end{document}